\documentclass{amsart}

\usepackage{macros}
\standardsettings
\colorcommentstrue

\newcommand{\FS}{\operatorname{FS}}
\newcommand{\FP}{\operatorname{FP}}
\newcommand{\synd}{s}
\renewcommand{\Neur}{{\N_0}}

\theoremstyle{theorem}
\maketheorem{maintheorem}{Main Theorem}

\draftfalse

\begin{document}
\title[Examples of complete sets]{New examples of complete sets, with connections to a Diophantine theorem of Furstenberg}

\authorvitaly\authordavid

\date{\today}

\begin{Abstract}
A set $A\subseteq\mathbb N$ is called \emph{complete} if every sufficiently large integer can be written as the sum of distinct elements of $A$. In this paper we present a new method for proving the completeness of a set, improving results of Cassels ('60), Zannier ('92), Burr, Erd{\doubleacute o}s, Graham, and Li ('96), and Hegyv\'ari ('00). We also introduce the somewhat philosophically related notion of a \emph{dispersing} set and refine a theorem of Furstenberg ('67).
\end{Abstract}
\maketitle

\section{Introduction}
For each $a,b\in\N = \{1,2,\ldots\}$ such that $a,b \geq 2$, let $\Gamma(a,b)$ denote the multiplicative semigroup generated by $a$ and $b$:
\begin{equation}
\label{Gammaabdef}
\Gamma(a,b) = a^\Neur b^\Neur = \{a^n b^m : n,m\in\Neur\},
\end{equation}
where $\Neur = \N\cup\{0\}$. This short note is dedicated to the refinement and generalization of two classical results which involve sets of the form $\Gamma(a,b)$. In order to formulate these results, we first need to introduce some notation and terminology.

\begin{definition}
For each set $A\subset\N$, we define the \emph{finite sum set} of $A$:
\[
\FS(A) = \left\{\left.\Sigma(F) := \sum_{n\in F} n  \right\vert \emptyset \neq F\subset A \text{ finite}\right\}.
\]
The set $A$ is called \emph{complete} if $\FS(A)$ is cofinite in $\N$, i.e. if $\#(\N\butnot\FS(A)) < \infty$.
\end{definition}

\begin{definition}
A set $A\subset\N$ is called \emph{dispersing} if for every irrational $\alpha \in \T := \R/\Z$, the set $A\alpha = \{n\alpha : n\in A\}$ is dense in $\T$.
\end{definition}

The word ``completeness'' was originally used to refer to a slightly different concept; namely, the set $\FS(A)$ was required to equal $\N$ rather than to merely be cofinite in it. This definition appeared first in a problem asked by Hoggatt and King and answered by Silver \cite{HKS}, and later the same year in a paper of Brown \cite{Brown}. It seems that Graham \cite{Graham} was the first to use the word ``completeness'' in the same (now standard) way that we use it.

By contrast, the notion of a ``dispersing'' set has not appeared explicitly in the literature before. It bears some resemblance to the notion of a ``Glasner set'' (cf. \cite{Glasner, BerendPeres}, and see \cite{BerendBoshernitzan} for a generalization).\Footnote{A set $A$ is called a \emph{Glasner set} if for every infinite set $I\subset\T$ and for every $\epsilon > 0$, there exists $n\in A$ such that the set $nI$ is $\epsilon$-dense.} However, the differences between these definitions are significant, and we will not discuss Glasner sets in this paper.

Although their definitions are very different, the notions of completeness and dispersion do share some relation. Both describe some notion of ``largeness'' of a set of integers which measures not just the growth rate but also in some sense the arithmetical properties of the set in question. This is manifested in the following classical results about complete and dispersing sets, which are due to Birch and Furstenberg, respectively:

\begin{theorem}[\cite{Birch2}]
\label{theorembirch}
For any coprime integers $a,b\in\N$ such that $a,b \geq 2$, the set $\Gamma(a,b)$ is complete.
\end{theorem}

\begin{theorem}[{\cite[Theorem IV.1]{Furstenberg3}}]
\label{theoremfurstenberg}
Fix $a,b\in\N$ with $a,b\geq 2$ and assume that $a,b$ are not powers of a single integer. Then $\Gamma(a,b)$ is dispersing.
\end{theorem}

These theorems indicate that some sort of ``semigroup property'' is useful for proving both completeness and dispersing results. However, on its own the semigroup property is not enough. Indeed, for any $a\in\N$, $a\geq 3$, the cyclic semigroup $\Gamma(a) = a^\Neur = \{a^n : n\in\Neur\}$ is neither complete nor dispersing: since $\FS(\Gamma(a))$ contains only those numbers whose base $a$ expansion consists of zeros and ones, $\FS(\Gamma(a))$ is of density zero (so $\Gamma(a)$ is incomplete), while if $\alpha\in\T$ is an irrational whose base $a$ expansion is missing some digit, then the set $\Gamma(a) \alpha$ is nowhere dense in $\T$ (so $\Gamma(a)$ is not dispersing). So it makes sense to augment the semigroup property with some information on the size of the set in question: the sets $\Gamma(a,b)$ are larger than the sets $\Gamma(a)$, and in general it is easier for larger sets to be complete and dispersing. In the case of the sets $\Gamma(a,b)$, information on the size is provided by the following lemma due to Furstenberg:

%


\begin{theorem}[{\cite[Lemma IV.1]{Furstenberg3}}]
Fix $a,b\in\N$ with $a,b\geq 2$ and assume that $a,b$ are not powers of a single integer. Then if we write
\[
\Gamma(a,b) = \{n_1,n_2,\ldots\}
\]
with $n_1 < n_2 < \cdots$, then
\begin{equation}
\label{sublacunary}
\lim_{k\to\infty} \frac{n_{k + 1}}{n_k} = 1.
\end{equation}
\end{theorem}

An increasing sequence $(n_k)_1^\infty$ satisfying \eqref{sublacunary} is called \emph{sublacunary}. By extension, the corresponding set $\{n_1,n_2,\ldots\}$ is also called sublacunary.

\begin{remark}
When interpreted as a condition on sequences, sublacunarity is an ``upper bound'' on the growth rate, but when interpreted as a condition on sets, sublacunarity is a ``lower bound'' on the size of a set (i.e. any set which contains a sublacunary set is also sublacunary).
\end{remark}

Just as that the semigroup property was seen to be insufficient without the sublacunarity property, so also the sublacunarity property is not enough to guarantee that a set is complete or dispersing without an additional property. We illustrate this fact by the following simple examples:

\begin{example}
\label{exampleNCD0}
Let $\alpha\in\T$ be an irrational number, and let $A = \{n \text{ even} : n\alpha\notin U\}$, where $U \subset \T$ is a non-dense open subset of $\T$. Then $A$ is sublacunary but neither complete nor dispersing. Indeed, it is clear that $\FS(A)$ contains only even numbers, while $A\alpha$ is disjoint from $U$ and therefore not dense. On the other hand, since $U$ is non-dense, $A$ is syndetic and therefore sublacunary. (Recall that a set $S\subset\N$ is called \emph{syndetic} if there exists a number $\synd\in\N$ (the \emph{syndeticity constant}\Footnote{Technically, we should say that the syndeticity constant is the \emph{smallest} number $\synd\in\N$ satisfying this condition.}) such that for all $n\in\N$, there exists $i = 0,\ldots,\synd$ such that $n + i \in S$.\Footnote{Sets which are syndetic according to our terminology are sometimes called \emph{syndetic in $\N$}, to distinguish them from sets which are syndetic in $\Z$. Since we deal only with sets which are syndetic in $\N$, we abbreviate by omitting the phrase ``in $\N$''. A similar comment applies to our definition of Bohr sets below.\label{footnoteinN}})
\end{example}

\begin{example}
\label{exampleNCD}
For each $\alpha\in\T$, we let $\|\alpha\|$ denote the distance in $\T$ from $\alpha$ to $0$, or equivalently the distance from any representative of $\alpha$ to the nearest integer. Fix a badly approximable\Footnote{An irrational $\alpha\in\T$ is called \emph{badly approximable} if there exists $\epsilon > 0$ such that for all $q\in\N$, we have $\|q\alpha\|\geq \epsilon/q$, or equivalently if the continued fraction expansion of $\alpha$ has bounded entries.} irrational $\alpha\in\T$, and for each $k\in\N$ let $k^3 \leq n_k < (k + 1)^3$ be chosen so as to minimize $\|n_k \alpha\|$. By a standard result in Diophantine approximation \cite[Theorem 26]{Khinchin_book}, we have $\|n_k \alpha\| \leq C[(k + 1)^3 - k^3]^{-1}$, where $C > 0$ is a large constant depending on $\alpha$. Choose $k_0\in\N$ large enough so that
\[
\sigma := \sum_{k = k_0}^\infty C[(k + 1)^3 - k^3]^{-1} < 1/2,
\]
and let $A = \{n_{k_0},n_{k_0 + 1},\ldots\}$. Then $A$ is sublacunary but neither complete nor dispersing. Indeed, while the bounds $k^3 \leq n_k < (k + 1)^3$ guarantee that $A$ is sublacunary, the fact that $\FS(A)$ is disjoint from the positive density set $\{n\in\N : \|n\alpha\| > \sigma\}$ implies that $A$ is not complete, and the fact that $\|n_k \alpha\| \to 0$ implies that $A\alpha$ is nowhere dense, so $A$ is not dispersing.
\end{example}

\begin{remark}
\label{remarksyndetic}
Although the set $\FS(A)$ of Examples \ref{exampleNCD0} and \ref{exampleNCD} is not cofinite, it is syndetic. The syndeticity of $\FS(A)$ for every sublacunary set $A$ follows from Lemma \ref{lemmasyndetic} below, which is a result due to Burr and Erd\doubleacute os \cite[Lemma 3.2]{BurrErdos}. However, Examples \ref{exampleNCD0} and \ref{exampleNCD} shows that cofiniteness of $\FS(A)$ is a much subtler matter.
\end{remark}


Examples \ref{exampleNCD0} and \ref{exampleNCD} notwithstanding, we will show in this paper that certain rather small subsets of $\Gamma(a,b)$ (or of more general multiplicative subsemigroups of $\N$) can be shown to be complete and/or dispersing. We conclude this introduction with a summary of the results obtained in this paper. (The proofs will be provided in the subsequent sections.)

\begin{convention*}
From now on, numerical variables (usually lowercase Latin letters) are assumed to take values in $\N$, and set variables (usually uppercase Latin letters) are assumed to take values which are subsets of $\N$, unless otherwise specified.
\end{convention*}

\begin{convention*}
If $\ast$ is an operation and $A,B$ are sets, then
\[
A\ast B := \{a\ast b : a\in A, \; b\in B\}.
\]
We may abbreviate $\{a\}\ast B$ by $a\ast B$ and $A\ast \{b\}$ by $A\ast b$. For example, $a^S = \{a^n : n\in S\}$. Note that this convention was already used implicitly in formula \eqref{Gammaabdef} when we wrote $\Gamma(a,b) = a^\Neur b^\Neur$.
\end{convention*}

{\bf Acknowledgements.} The authors thank Trevor Wooley for directing them to the paper of Freeman cited later in this paper. The first-named author was supported by NSF grant DMS-1162073. The authors thank the anonymous referee for valuable comments.

\subsection{Completeness results}
To motivate our first result, we recall a remark in Birch's paper which he attributes to Davenport \cite[para. after Theorem]{Birch2}, namely that the proof of Theorem \ref{theorembirch} in that paper can be strengthened to demonstrate the following ``finitary'' version of the theorem:
\begin{theorem}[Davenport's remark]
\label{theoremdavremark}
For every $a,b\geq 2$ such that $\gcd(a,b) = 1$, there exists $N\in\N$ such that the set
\[
\{a^n b^m : n,m\in\Neur , \; m\leq N\}
\]
is complete.
\end{theorem}
A quantitative version of Theorem \ref{theoremdavremark} was proven by Hegyv\'ari \cite{Hegyvari}. We will strengthen Theorem \ref{theoremdavremark} by replacing the expression $b^m$ by an arbitrary expression depending on $m$, subject to some mild restraints, which should be thought of as the analogue of the condition $\gcd(a,b) = 1$. At the same time we will improve Hegyv\'ari's result by giving a better quantitative bound on the number $N$. Precisely, we have the following:

\begin{theorem}
\label{completenesstheorem1fin}
Fix $a\geq 2$, and let $(b_m)_0^N$ be a finite sequence such that
\begin{itemize}
\item[(I)] The numbers $(\log_a(b_m))_0^N$ are distinct mod 1.
\item[(II)] $\gcd(b_0,b_1,\cdots,b_{N - 3(a - 1)}) = 1$.
\item[(III)] $\#\{m = 0,\ldots,N - 3(a - 1) : \gcd(a,b_m) = 1\} \geq a - 1$.
\end{itemize}
Then the set
\begin{equation}
\label{completeness1fin}
A = \{a^n b_m : n,m\in\Neur , \; m\leq N\}
\end{equation}
is complete.
\end{theorem}

Theorem \ref{theoremdavremark} corresponds to the special case $b_m = b^m$, where $\gcd(a,b) = 1$. In this case, the conditions (I)-(III) are satisfied when $N = 4a - 5$, which vastly improves the fourfold-exponential bound of \cite{Hegyvari}. Taking the slightly more general special case $b_m = b^{k_m}$ yields the following corollary (which implies the aforementioned improvement of Hegyv\'ari's result):


\begin{corollary}
\label{completenesscorollary1}
Fix $a,b\geq 2$ coprime and let $(k_m)_0^{4a - 5}$ be a finite sequence of distinct integers such that $k_0 = 0$. Then the set
\[
\{a^n b^{k_m} : n,m\in\Neur , \; m\leq 4a - 5\}
\]
is complete.
\end{corollary}

Finally, we also state an infinitary version of Theorem \ref{completenesstheorem1fin}:

\begin{corollary}
\label{completenesstheorem1}
Fix $a\geq 2$, and let $(b_m)_0^\infty$ be a sequence such that
\begin{itemize}
\item[(I)] The sequence $(\log_a(b_m))_0^\infty$ contains infinitely many distinct elements mod 1.
\item[(II)] $\gcd(b_0,b_1,\ldots) = 1$.
\item[(III)] $\#\{m\in\Neur : \gcd(a,b_m) = 1\} \geq a - 1$.
\end{itemize}
Then for some $N\in\N$, the set $A$ defined by \eqref{completeness1fin} is complete.
\end{corollary}

Although the set $A$ defined by \eqref{completeness1fin} of Theorem \ref{completenesstheorem1fin} is not a semigroup, it contains the semigroup $\Gamma(a)$, and indeed can be decomposed as the product of $\Gamma(a)$ with the finite set $\{b_m : m = 0,\ldots,N\}$. This multiplicative structure is used somewhat as a substitute for the semigroup property in the proof of Theorem \ref{completenesstheorem1fin}. It is interesting to ask how much this multiplicative structure can be weakened without losing completeness. For example, is the decomposition of the set as the product of two ``nice'' sets enough? The following example shows that even in the best of circumstances (short of the semigroup property in one of the factors), a single product decomposition is not enough to guarantee completeness:

\begin{example}
\label{exampleanbm}
Fix $a,b\geq 2$. Then the set
\[
A = \{a^{n^2} b^{m^2} : n,m\in\Neur\}
\]
is not complete. Indeed, an analysis of growth rates (cf. \6\ref{proofanbm}) shows that the set $\FS(A)$ has density zero.
\end{example}

To counteract the phenomenon in this example, we can include more multiplicative structure by increasing the number of factors allowed without changing their form. For example, given a finite sequence $(a_i)_1^s$, we can consider the set
\[
\{a_1^{n_1^2}\cdots a_s^{n_s^2} : n_1,\ldots,n_s\in\Neur\}.
\]
Our next theorem shows that if $s\geq 6$ and $(a_i)_1^s$ are pairwise coprime, then this set is complete. Let $\PP_\N$ denote the collection of nonconstant polynomials $P$ such that $P(\Neur)\subset\Neur$ and $P(0) = 0$. For each $k$, let $\PP_\N^k$ denote the collection of polynomials in $\PP_\N$ of degree $\leq k$.

\begin{theorem}
\label{completenesstheorem2}
For all $k\geq 2$ there exists $s = s_0(k)\in\N$ such that for all $a_1,\ldots,a_s\geq 2$ and $P_1,\ldots,P_s\in\PP_\N^k$, if
\begin{itemize}
\item[(I)] $\gcd(a_1,\ldots,a_s) = 1$, and
\item[(II)] $\log(a_1),\ldots,\log(a_s)$ are linearly independent over $\Q$,
\end{itemize}
then the set
\begin{equation}
\label{completeness2}
A_s := \left\{\prod_{i = 1}^s a_i^{P_i(n_i)} : n_1,\ldots,n_s\in\Neur\right\}
\end{equation}
is complete. Moreover, we may take $s_0(k)$ to satisfy
\begin{equation}
\label{s0bounds}
s_0(k) \sim 8k\log(k), \;\; s_0(2) = 6.
\end{equation}
\end{theorem}

\begin{remark}
\label{remarks0bounds}
In addition to the upper bounds \eqref{s0bounds}, we can also give the following lower bounds:
\[
s_0(k) \geq k, \;\; s_0(2) \geq 3.
\]
Both of these bounds follow from growth rate calculations; see \6\ref{proofs0bounds} for the first bound and \6\ref{proofanbm} for the second bound. It seems like a difficult problem to give better bounds on the function $s_0$.
\end{remark}

\begin{remark}
The general theorem which we use to prove our completeness results (i.e. Theorem \ref{maintheorem} below) is somewhat similar to a theorem of Cassels \cite{Cassels2}, about which we will say more later. While Cassels' result is not strong enough to deduce Theorem \ref{completenesstheorem1fin}, its corollaries, or the theorems which we state below, it is strong enough to prove Theorem \ref{completenesstheorem2} (possibly with a worse value of $s_0(k)$) via the theorem of Freeman mentioned above. We omit the details of this derivation, as the proof of Theorem \ref{completenesstheorem2} we give will be based on our own main theorem.
\end{remark}

Our next result is a generalization of a theorem of Zannier \cite{Zannier}. Zannier observed that Cassels' aforementioned result implies that if $P$ is a polynomial function (possibly with real coefficients), then the set
\begin{equation}
\label{prezannier}
A = \{\lfloor P(n) \rfloor : n\in\N\}
\end{equation}
is complete as long as $\gcd(A) = 1$.\Footnote{We remark that when $k = \gcd(A) > 1$, then $A = kB$ for some set $B$ of the same form as $A$ which satisfies $\gcd(B) = 1$. Consequently, $\FS(A)$ is cofinite in $k\N$.} He then used elementary methods to prove another completeness theorem which implies this statement. We are now able to generalize Zannier's theorem as follows:

\begin{theorem}
\label{theoremzannier}
Let $A$ be a sublacunary set, and suppose that there exist $z_1,\ldots,z_k\in\Z$ and $b\in\N$ such that
\begin{itemize}
\item[(I)] for all $N\in\N$, there exist $x_1,\ldots,x_k \in A$ such that $x_i \geq N \all i \leq k$ and
\begin{equation}
\label{zannier}
0 < \left|\sum_{i = 1}^k z_i x_i \right| \leq b;
\end{equation}
\item[(II)] for all $q = 2,\ldots,b$, $\FS(A)$ intersects every arithmetic progression of the form $q\N + i$ $(0 \leq i < q)$.
\end{itemize}
Then $A$ is complete.
\end{theorem}

We now state Zannier's result and deduce it as a corollary of Theorem \ref{theoremzannier}:

\begin{corollary}[Main theorem of \cite{Zannier}]
\label{corollaryzannier}
Let $A$ be a sublacunary set and let $(x(i))_1^\infty$ be its unique increasing indexing, and suppose that there exist $z_1,\ldots,z_\ell\in\Z$ and $b\in\N$ such that
\begin{itemize}
\item[(I)] there exists $c > 0$ such that for all $N\in\N$, there exist $\ell$-tuples $(i_1,\ldots,i_\ell)$ and $(j_1,\ldots,j_\ell)$ such that $i_m \geq j_m \geq N \all m\leq \ell$, and the following hold:
\begin{align} \tag{$\alpha$}
&x(i_m)/x(j_m) \to 1 \all m \leq \ell \text{ as $N\to\infty$} \\ \tag{$\beta$}
&x(i_\ell) \leq c x(i_1) \\ \tag{$\gamma$}
&0 < \left|\sum_{m = 1}^\ell z_m (x(i_m) - x(j_m)) \right| \leq b;
\end{align}
\item[(II)] for all $q = 1,\ldots,b$, $\FS(A)$ intersects every arithmetic progression of the form $q\N + i$ ($0 \leq i < q$).
\end{itemize}
Then $A$ is complete.
\end{corollary}
\begin{proof}
Let $k = 2\ell$, $z_{\ell + m} = -z_m$ ($m = 1,\ldots,\ell$), $x_m = x(i_m)$, $x_{\ell + m} = x(j_m)$ ($m = 1,\ldots,\ell$) in Theorem \ref{theoremzannier}.
\end{proof}

Actually, this proof shows that in Corollary \ref{corollaryzannier}, the conditions $(\alpha)$ and $(\beta)$ are both unnecessary.

Another application of Theorem \ref{theoremzannier} is that it is used in the proof of the following result:

\begin{theorem}
\label{theoremnalpha}
Fix $k$ and let $f:\N\to\N$ be a function whose $k$th difference $\Delta^k f$ is bounded, where
\[
\Delta f(n) = f(n + 1) - f(n).
\]
Then if
\[
A = \{\lfloor f(n)\rfloor : n\in\N\},
\]
then $\FS(A)$ contains an arithmetic progression.
\end{theorem}
For example, we could take $f(x) = x^\alpha$, where $\alpha > 0$ is an irrational number. Note that if $f:(0,\infty)\to (0,\infty)$ is a $C^k$ function whose $k$th derivative is bounded, then $\Delta^k f$ is also bounded.

Another way to generalize Zannier's result is to consider the images of ``sufficiently large'' sets under polynomial mappings. It turns out that the lower bound on the size of the set of primes guaranteed by the prime number theorem is enough to show that the image of the set of primes under any arithmetically appropriate polynomial mapping is complete. We phrase this result more generally as follows:

\begin{theorem}
\label{theoremdPD}
Fix $d\in\N$ and $P\in\P_\N^d$, let $D$ be a sublacunary set such that
\begin{equation}
\label{casselsassumption2mod}
\liminf_{N\to\infty} \frac1{N^{1 - \delta}} \#(D\cap [1,N]) > 0.
\end{equation}
where $\delta = 1/\left[1 + 2\binom{d + 1}2\right]$. Let $A = P(D)$, and assume that for all $q\geq 2$,
\begin{equation}
\label{casselsassumption3mod}
\#\{n\in A : q \not\divides n\} = \infty.
\end{equation}
Then $A$ is strongly complete.
\end{theorem}

Here a set is said to be \emph{strongly complete} if it remains complete after removing any finite subset.

\begin{corollary}
\label{corollarypolyprime}
Let $D$ denote the set of primes. Fix $d\in\N$ and $P\in\P_\N^d$ such that for all $q\geq 2$, \eqref{casselsassumption3mod} holds for $A = P(D)$. Then $A$ is strongly complete.
\end{corollary}
\begin{proof}
The prime number theorem guarantees that the set of primes is sublacunary and satisfies \eqref{casselsassumption2mod}.
\end{proof}

In particular, Corollary \ref{corollarypolyprime} reproves a result of Roth and Szekeres \cite[sequence (iii) on p.241]{RothSzekeres}. Moreover, letting $P(x) = x$ shows that the set of primes is strongly complete. This result can be compared to Goldbach's conjecture, in the sense that it states that any sufficiently large number can be written as the sum of (a possibly large number of) large primes, whereas Goldbach's conjecture claims that any number $\geq 4$ can be written as the sum of at most three primes.

Our last result regarding completeness is a generalization of a theorem of Burr, Erd\doubleacute os, Graham, and Li \cite{BEGL}. These authors propose a different way of weakening the semigroup property while keeping some multiplicative structure, by considering the completeness of unions of sets of the form $\Gamma(a)$. They go on to conjecture that for $S\subset\N\butnot\{1\}$ such that no two elements of $S$ are powers of the same integer,\Footnote{Although the authors of \cite{BEGL} do not state this assumption explicitly, it is necessary to translate between the language of ``sequences'' used in their paper (which seem to really be multisets) and the set-theoretic language used in this paper. If $a,a^2\in S$, then they seem to allow $a^{2n}$ and $(a^2)^n$ to appear as separate terms in a decomposition of an element of $\FS(S^\Neur)$ (which the authors of \cite{BEGL} denote $\mathrm{Pow}(S;0)$), whereas it is a consequence of our notation that we do not consider such decompositions legal.} the set
\[
S^\Neur = \bigcup_{a\in S} \Gamma(a)
\]
is strongly complete if and only if $\gcd(S) = 1$ and
\begin{equation}
\label{syndeticityequiv}
\sum_{a\in S} \frac{1}{a - 1} \geq 1.
\end{equation}
While we can neither prove nor disprove this conjecture, the following result generalizes the main theorem of \cite{BEGL}:

\begin{theorem}
\label{theoremBEGL}
Let $S_1,S_2,S_3,S_4\subset\N\butnot\{1\}$ be finite pairwise disjoint sets such that $\gcd(S_4) = 1$, and for each $i = 1,2,3$
\begin{equation}
\label{a1sum}
\sum_{a\in S_i} \frac{1}{a - 1} \geq 1.
\end{equation}
Then the set $A = S^\Neur$ is strongly complete, where $S = \bigcup_1^4 S_i$.
\end{theorem}

\begin{corollary}[Main theorem of \cite{BEGL}]
\label{corollaryBEGL}
Let $S\subset\N\butnot\{1\}$ be a set such that
\[
\limsup_{N\to\infty} \frac{1}{N}\#(S\cap [1,N]) > 0
\]
and $\gcd(S) = 1$. Then the set $A = S^\Neur$ is strongly complete.
\end{corollary}

Corollary \ref{corollaryBEGL} is deduced from Theorem \ref{theoremBEGL} by decomposing the set $S$ appropriately, and throwing out an infinite component. However, Theorem \ref{theoremBEGL} applies in many circumstances where Corollary \ref{corollaryBEGL} does not apply; for example, Theorem \ref{theoremBEGL} applies to some finite sets $S$, whereas Corollary \ref{corollaryBEGL} applies only to infinite sets $S$. The hypotheses of Theorem \ref{theoremBEGL} are still significantly stronger than the conjectured \eqref{syndeticityequiv}, which is known to be the necessary and sufficient condition for $\FS(S^\Neur)$ to be syndetic. This illustrates the great difference between syndeticity and cofiniteness for sets of the form $\FS(A)$, at least in terms of our knowledge about them.

As another illustration of this difference, we include the following observation, which also offers a nice transition to our discussion of the dispersing condition:

\begin{proposition}
\label{completenesstheorem0fin}
Fix $a,b\geq 2$, not both powers of the same integer. Let $S\subset\N$ be a syndetic set and let $T\subset\N$ be a set of cardinality at least $a^m - 1$, where $m$ is the syndeticity constant of $S$. Then $\FS(a^S b^T)$ is syndetic.
\end{proposition}

\subsection{Dispersing results}
The dispersing condition seems to be heuristically somewhat stronger than the completeness condition. While in Corollary \ref{completenesscorollary1} we were able to replace the sequence $(b^m)_0^\infty$ in the definition of $\Gamma(a,b)$ by any sequence of the form $(b^{k_m})_0^\infty$ such that $k_0 = 0$, getting a similar result regarding dispersing sets appears to require a condition on the sequence $(k_m)_0^\infty$. Our first result is that it is sufficent that the set $\{k_0,k_1,\ldots\}$ is piecewise syndetic. We recall the definition of this condition as well as some related definitions:

\begin{definition}
A set $S\subset\N$ is called \emph{thick} if it contains arbitrarily large intervals, and \emph{piecewise syndetic} if it is the intersection of a thick set with a syndetic set (cf. Remark \ref{remarksyndetic}). A set $S$ is called \emph{Bohr}\Footnote{Cf. Footnote \ref{footnoteinN}.} if there exist $d\in\N$, $\alpha\in \T^d = \R^d/\Z^d$, and an open set $\emptyset\neq U\subset \T^d$ such that
\[
\emptyset \neq \{n\in\N : n\alpha\in U\} \subset S.
\]
Finally, the intersection of a thick set with a Bohr set is called \emph{piecewise Bohr}.
\end{definition}

To state our results more concisely, it will help to introduce some new terminology regarding variants of the dispersing condition.
\begin{definition}
Fix $\epsilon > 0$. A set $A\subset\N$ is \emph{$\epsilon$-dispersing (resp. weakly dispersing)} if for every irrational $\alpha\in \T$, the set $A\alpha$ is $\epsilon$-dense (resp. somewhere dense\Footnote{A set is called \emph{somewhere dense} if it is not nowhere dense, i.e. if its closure contains a nonempty open set.}) in $\T$.
\end{definition}


\begin{theorem}
\label{dispersingtheorem0}
Fix $a,b\geq 2$ not both powers of the same integer. Let $S$ be a syndetic set and let $T$ be a piecewise syndetic set. Then the set
\[
a^S b^T
\]
is weakly dispersing.
\end{theorem}


Since the product of an infinite subset of $\N$ with a nonempty open subset of $\T$ is equal to $\T$, the product of an infinite set with a weakly dispersing set is dispersing. Thus we deduce the following corollary:

\begin{corollary}
\label{dispersingcorollary0}
Fix $a,b\geq 2$ not both powers of the same integer. Let $S$ be a syndetic set, let $T$ be a piecewise syndetic set, and let $I$ be an infinite set. Then the set
\[
a^S b^T I
\]
is dispersing.
\end{corollary}

Considering the case where $I$ takes the form $a^J$ gives another corollary:


\begin{corollary}
Fix $a,b\geq 2$ not both powers of the same integer. Let $S$ be a Bohr set and let $T$ be a piecewise syndetic set. Then the set
\[
a^S b^T
\]
is dispersing.
\end{corollary}
\begin{proof}
Since $S$ is Bohr, it contains a set of the form $S_1 + S_2$, where $S_1,S_2$ are both Bohr. In particular, $S_1$ is syndetic and $S_2$ is infinite, so applying Corollary \ref{dispersingcorollary0} completes the proof.
\end{proof}

Although piecewise syndetic sets can be made to grow at an arbitrarily slow rate, they are still in some sense ``large'' because they have large pieces. It is possible to substitute this largeness by an additional additive structure hypothesis on $T$. Specifically, if $T$ is the finite sum set of a set $R\subset\N$ with certain arithmetical properties, then $a^S b^T$ is dispersing:

\begin{theorem}
\label{dispersingtheorem2}
Fix $a,b\geq 2$ not both powers of the same integer. Let $S$ be a syndetic set and let $T = \FS(R)$, where $R$ is a set such that for all $k$, $(R/k \cap \N)\log_b(a)$ is dense mod one. Then the set
\[
a^S b^T
\]
is dispersing.
\end{theorem}

Note that the hypothesis given on $R$ imposes no restriction on how slowly $R$ grows; if $f:\N\to\N$ is any function, then we may choose $R = \{n_1,n_2,\ldots\}$ to satisfy $n_{k + 1} \geq f(n_k) \all k$. So for example, by choosing $R$ appropriately we can make the upper Banach density of $T$ equal to zero.\Footnote{Recall that the \emph{upper Banach density} of a set $T\subset\N$ is the number \[d^*(T) = \limsup_{N\to\infty} \frac1N \sup_{M\in\N} \#(T\cap [M,M + N]),\] which satisfies $d^*(T) > 0$ whenever $T$ is piecewise syndetic.}


Next we consider a dispersing analogue of Theorem \ref{completenesstheorem2}. Again the dispersing condition appears to be stronger than the completeness condition: to get a set which we can prove to be dispersing, we need to take the union over all $s$ of a sequence of sets of the form \eqref{completeness2}.

\begin{theorem}
\label{dispersingtheorem1}
Let $(a_i)_1^\infty$ be an infinite sequence of integers, no two of which are powers of the same integer, and suppose there exists a prime $p$ such that the set $\{a_i : p \text{ does not divide } a_i\}$ is infinite. Fix $k\geq 2$ and a sequence $(P_i)_1^\infty$ in $\PP_\N^k$ (cf. Theorem \ref{completenesstheorem2}). For each $s\in\N$ let $A_s$ be given by \eqref{completeness2}. Then the set $A = \bigcup_1^\infty A_s$ is dispersing. More precisely, for every $\epsilon > 0$ there exists $s$ such that the set $A_s$ is $\epsilon$-dispersing.
\end{theorem}

It appears to be a difficult question whether or not the sets $A_s$ in Theorem \ref{dispersingtheorem1} are dispersing for sufficiently large $s$. This may make the theorem seem trivial on some level, because the final set $A$ is decomposed as the product of infinitely many infinite sets. But by itself this property is not enough to guarantee dispersing, as shown by the following theorem:

\begin{theorem}
\label{dispersingtheorem4}
Let $(a_i)_1^\infty$ be a sequence of integers such that $a_i\geq 2$ for all $i$. Then there exist thick sets $(S_i)_1^\infty$ such that the set
\[
A = \prod_{i = 1}^\infty \{1\}\cup a_i^{S_i} = \bigcup_{F\subset\N} \prod_{i\in F} a_i^{S_i}
\]
is not weakly dispersing.
\end{theorem}

Theorem \ref{dispersingtheorem4} can be interpreted as saying that an infinite multiplicative decomposition property is not enough to replace the semigroup property, while Theorem \ref{dispersingtheorem1} says that it is enough if the sets $S_i$ have an algebraic structure. The next theorem does not require further algebraic structure of the factors of an infinite multiplicative decomposition, but only requires a growth condition (sublacunarity) as well as a divisibility condition.

\begin{theorem}
\label{dispersingtheorem3}
Let $S$ be a set with the following property: there exist infinitely many $r\in\N$ such that $S\cap (r\N + 1)$ is sublacunary. Then the finite product set
\[
\FP(S) := \left\{\Pi(F) = \prod_{n\in F} n : F\subset S \text{ finite}\right\}
\]
is dispersing.
\end{theorem}

From what we have said so far, it might appear that it is always harder to prove a dispersing result than a corresponding completeness result, or even that the dispersing property might somehow imply the completeness one. But this is not true, as we can show in two different ways. First of all, if $a,b\geq 3$ are not powers of the same integer but $\gcd(a,b) \geq 3$ (e.g. $a = 3$, $b = 6$), then by Theorem \ref{theoremfurstenberg} the set $\Gamma(a,b)$ is dispersing, but it follows from arithmetic considerations that $\Gamma(a,b)$ is not complete. Second, and more significantly, the completeness property is tied to growth rates in a way that the dispersing property is not. If a set $A$ is complete, then a counting argument implies that
\[
\#\{n\in A : n \leq 2^N + \synd\} \geq N \all N\in\N
\]
for some constant $\synd\in\N$. By contrast, the following observation shows that there is no lower bound on the growth rates of dispersing sets:

\begin{observation}
Every thick set is dispersing, and every piecewise syndetic set is weakly dispersing.
\end{observation}
\begin{proof}
Let $A\subset\N$ be a thick set. Then there exists a sequence $n_k\to\infty$ such that $A \supset \{n_k + m : 0 \leq m \leq k\}$ for all $k$. Fix $\alpha\in \T$ irrational and $\epsilon > 0$. Then for some $k$, the set $\{0,\alpha,\ldots,k\alpha\}$ is $\epsilon$-dense in $\T$. By adding $n_k \alpha$, we see that $A\alpha$ is $\epsilon$-dense in $\T$.

If $A\subset\N$ is piecewise syndetic, then $A + F$ is thick for some finite set $F\subset\N$. If $\alpha\in \T$ is irrational, then $\cl{A\alpha} + F\alpha = \T$ by the above argument, so by elementary topology, one of the sets $\cl{A\alpha} + i\alpha$ ($i\in F$) contains a nonempty open set. Thus $A\alpha$ is somewhere dense.
\end{proof}

This observation is ``optimal'' in the sense that not every syndetic set is dispersing, and no lower bound on the growth rate of a set weaker than syndeticity is sufficient to guarantee that a set is weakly dispersing. More precisely, given any $\alpha > 0$ the syndetic set
\[
\{n\in\N : \|n\alpha\| < 1/4\}
\]
is not dispersing, and the following observation shows that any ``growth rate lower bound'' which is satisfied for some density zero set is also satisfied for some set which is not weakly dispersing:

\begin{observation}
\label{observationzerodensity}
Let $(m_k)_1^\infty$ be an increasing sequence of integers such that $m_{k + 1} - m_k \to \infty$, and fix $\beta\in\T$. Then for all irrational $\alpha\in\T$ there exists a sequence $(n_k)_1^\infty$ such that $\|n_k \alpha\| \to \beta$ and for all $k$, $m_k \leq n_k < m_{k + 1}$. In particular, $\{n_k : k\in\N\}$ is not weakly dispersing.
\end{observation}
\begin{proof}
Choose $n_k \in \{m_k,\ldots,m_{k + 1} - 1\}$ so as to minimize $\|n_k \alpha - \beta\|$. If $\{0,\ldots,N\}\alpha$ is $\epsilon$-dense mod 1 and $m_{k + 1} - m_k > N$, then $\|n_k\alpha - \beta\| \leq \epsilon$. Thus since $m_{k + 1} - m_k \to \infty$, we have $\|n_k\alpha - \beta\|\to 0$.
\end{proof}

The following corollary was also obtained by Porubsky and Strauch \cite{PorubskyStrauch}:

\begin{corollary}
Let $(\epsilon_k)_1^\infty$ be a decreasing sequence of real numbers such that $\epsilon_k \to 0$, and fix $\beta\in\T$. Then for all irrational $\alpha\in\T$ there exists a sequence $(n_k)_1^\infty$ such that $\|n_k \alpha\| \to \beta$ and $k/n_k \geq \epsilon_k$ for all $k$.
\end{corollary}
\begin{proof}
Take $m_k = \lceil k/\epsilon_k\rceil$ and apply the previous observation.
\end{proof}



{\bf Outline of the paper.} The proofs of all theorems regarding completeness will be given in Section \ref{sectioncompleteness}, while the proofs of all theorems regarding the dispersing condition will be given in Section \ref{sectiondispersing}. The Appendix contains auxiliary calculations regarding the remarks surrounding Theorem \ref{completenesstheorem2}.

\section{Proofs of completeness results}
\label{sectioncompleteness}

We begin by stating the main theorem we will use to prove our completeness results.

\begin{maintheorem}
\label{maintheorem}
Let $B_1,B_2,B_3,C\subset\N$ be four pairwise disjoint sets such that:
\begin{itemize}
\item[(I)] For all $i = 1,2,3$,
\begin{equation}
\label{casselsassumption1}
\sup\left\{\big(n - \sum\{m\in B_i : m < n\}\big) : n\in B_i\right\} < \infty.
\end{equation}
\item[(II)] For all $\alpha\in\T$ irrational,
\begin{equation}
\label{casselsassumption2}
\sum_{n\in C} \|n\alpha\| = \infty.
\end{equation}
\item[(III)] For all $q$,
\begin{equation}
\label{casselsassumption3}
\FS(C) + q\Z = \Z.
\end{equation}
\end{itemize}
Then $A = B_1\cup B_2\cup B_3\cup C$ is complete.
\end{maintheorem}

It is worth comparing this theorem to a theorem of Cassels:

\begin{theorem}[{\cite[Theorem I]{Cassels2}}]
\label{theoremrealcassels}
Fix $A\subset\N$. Suppose that
\begin{equation}
\label{realcassels1}
\lim_{N\to\infty} \frac{\#(A\cap [N + 1,2N])}{\log\log(N)} = \infty
\end{equation}
and that for every $\alpha\in\T$ such that $\alpha\neq 0$,
\begin{equation}
\label{realcassels2}
\sum_{n\in A} \|n\alpha\|^2 = \infty.
\end{equation}
Then $A$ is complete.
\end{theorem}

\begin{remark}
Theorem \ref{maintheorem} is close to being a generalization of Theorem \ref{theoremrealcassels}, but does not quite succeed at doing so due to a technical issue. To be more precise (and referring to Remarks \ref{remarksqrt62}, \ref{remarkqinduction}, and \ref{remarkqinduction2} below for details), any set satisfying the hypotheses of Theorem \ref{theoremrealcassels} automatically satisfies \eqref{casselsassumption2} and \eqref{casselsassumption3}, and can be written as the disjoint union of arbitrarily many sets satisfying \eqref{casselsassumption1}, but it is not clear whether the decomposition can be chosen so that any member of this union satisfies \eqref{casselsassumption2} and \eqref{casselsassumption3}. Nevertheless, in practice it is usually easy to decompose a set satisfying \eqref{casselsassumption1}-\eqref{casselsassumption3} as a disjoint union as in Theorem \ref{maintheorem}, so Theorem \ref{maintheorem} is a sort of ``functional generalization'' of Theorem \ref{theoremrealcassels}. The converse is not true, since many naturally occurring sets satisfy \eqref{casselsassumption1} but not \eqref{realcassels1}, such as the sets occurring in the introduction of this paper (with the exception of those occurring in Theorem \ref{completenesstheorem2}).
\end{remark}

Before proving Theorem \ref{maintheorem}, we discuss some methods for checking its hypotheses.

\begin{remark}
\label{remarksyndetic2}
To check that \eqref{casselsassumption1} holds for some set $B_i$, it suffices to check that
\begin{equation}
\label{syndetic2}
\liminf_{N\to\infty} \#\big(B_i\cap \OC N{(L + 1)N}\big) \geq L
\end{equation}
for some $L$.
\end{remark}
\begin{proof}
Let $(n_k)_1^\infty$ be the unique increasing indexing of $B$, and let $k_0$ be large enough so that for all $k\geq k_0$,
\[
n_k \leq (L + 1) n_{k - L} \leq n_{k - 1} + \ldots + n_{k - L + 1} + 2n_{k - L}.
\]
Then an induction argument shows that for all $k\geq k_0$,
\[
n_k \leq \sum_{i = k_0}^{k - 1} n_i + L n_{k_0}.
\qedhere\]
\end{proof}

In particular, to check that a given set $B$ can be decomposed as the union of three pairwise disjoint sets satisfying \eqref{casselsassumption1}, it suffices to check that
\begin{equation}
\label{casselsassumption1prime}
\liminf_{N\to\infty} \#\big(B\cap \OC N{(L + 1)N}\big) \geq 3L
\end{equation}
for some $L$. Thus we have the following corollary of Theorem \ref{maintheorem}:

\begin{corollary}
\label{mainlemmacorollary}
Let $B,C\subset\N$ be two disjoint sets satisfying \eqref{casselsassumption1prime}, \eqref{casselsassumption2}, and \eqref{casselsassumption3}. Then $B\cup C$ is complete.
\end{corollary}

Note that any sublacunary set automatically satisfies \eqref{casselsassumption1prime} for all $L$. In fact we can say more; for this purpose we introduce some new terminology. Given $\lambda > 1$, a sequence $(n_k)_1^\infty$ is called \emph{$\lambda$-sublacunary} if $n_{k + 1}/n_k \leq \lambda$ for all $k$ sufficiently large. Note that $(n_k)_1^\infty$ is sublacunary if and only if it is $\lambda$-sublacunary for all $\lambda > 1$. We call $(n_k)_1^\infty$ \emph{weakly sublacunary} if it is $\lambda$-sublacunary for some $\lambda > 1$. As before, a set is called $\lambda$-sublacunary or weakly sublacunary if its unique increasing indexing has that property. Then we have:

\begin{remark}
\label{remarksqrt62}
Any $\sqrt[3]2$-sublacunary set satisfies \eqref{casselsassumption1prime} with $L = 1$. In particular, this includes the class of sets satisfying \eqref{realcassels1}.
\end{remark}

When checking condition \eqref{casselsassumption2}, it is useful for $C$ to have some multiplicative structure in the form of a factorization:

\begin{remark}
\label{remarkAI}
If $C_1$ is a weakly sublacunary set and $C_2$ is an infinite set, then the set $C = C_1 C_2$ satisfies \eqref{casselsassumption2} for all irrational $\alpha\in\T$.
\end{remark}
\begin{proof}
Fix $m_0\in C_1$ and $\lambda > 1$ such that for all $m\geq m_0$, $(m,\lambda m)\cap C_1\neq\emptyset$. Fix $N\in\N$. By the pigeonhole principle, there exist $n_1,n_2\in C_2$, $n_1,n_2\geq N$, such that $\|(n_2 - n_1)\alpha\| \leq 1/(2 m_0)$. Let $m$ be the largest element of $C_1$ such that $\|(n_2 - n_1)\alpha\| \leq 1/(2m)$, and note that $m\geq m_0$. Then since $(m,\lambda m)\cap C_1\neq\emptyset$, the maximality of $m$ implies that
\[
\|m(n_2 - n_1)\alpha\| = m\|(n_2 - n_1)\alpha\| > 1/(2\lambda).
\]
Thus there exists $i = 1,2$ such that $\|m n_i\alpha\| > 1/(4\lambda)$. Since $N$ was arbitrary, there exist infinitely many $n\in C$ such that $\|n\alpha\| > 1/(4\lambda)$. This completes the proof.
\end{proof}

\begin{remark}
\label{remarkqinduction}
To check \eqref{casselsassumption3} it suffices to show that for all $q\geq 2$, there exists $r < q$ such that
\begin{equation}
\label{qinduction}
\FS(C\cap r\N) + q\Z = r\Z.
\end{equation}
\end{remark}
\begin{proof}
Suppose this holds, and fix $q\in\N$. Let $q = q_0 > q_1 \ldots > q_k = 1$ be a decreasing sequence such that for each $i = 0,\ldots,k - 1$,
\[
\FS(C\cap q_{i + 1}\N) + q_i\Z = q_{i + 1}\Z.
\]
Clearly, we also have
\[
\FS\big(C\cap (q_{i + 1}\N\butnot q_i\N)\big) + q_i\Z = q_{i + 1}\Z
\]
and thus
\[
\FS(C) + q\Z \supset \sum_{i = 0}^{k - 1} \FS\big(C\cap (q_{i + 1}\N\butnot q_i\N)\big) + q_0\Z = q_k\Z = \Z.
\qedhere\]
\end{proof}

\begin{remark}
\label{remarkqinduction2}
For fixed $r < q$, to check \eqref{qinduction} it suffices to show that
\begin{equation}
\label{qinduction2}
\#\{n\in C : \gcd(n,q) = r\} \geq q/r - 1.
\end{equation}
\end{remark}
\begin{proof}
Let $D = \{n\in C : \gcd(n,q) = r\} \subset C\cap r\N$ and write $D = \{n_1,\ldots,n_k\}$, where $k\geq q/r - 1$. For each $i = 0,\ldots,k$ write $S_i = \FS(\{n_1,\ldots,n_i\}) + q\Z$. Fix $i = 0,\ldots,k - 1$. If $S_i$ is forward invariant under translation by $n_{i + 1}$, then the condition $\gcd(n_{i + 1},q) = r$ guarantees that $S_i = r\Z$, completing the proof. Otherwise, there exists $m\in S_i$ such that $m + n_{i + 1}\notin S_i$, which implies that $\#(S_{i + 1}/q\Z) > \#(S_i/q\Z)$. Since $\#(S_0/q\Z) = 1$, an induction argument gives $\#(S_i/q\Z) \geq i + 1$ for all $i$, and in particular $S_{q/r - 1}/q\Z = r\Z/q\Z$, completing the proof.
\end{proof}

Combining with a pigeonhole argument yields the following:

\begin{remark}
\label{remarkqinduction3}
For fixed $q\geq 2$, to prove the existence of $r < q$ satisfying \eqref{qinduction} it suffices to show that
\begin{equation}
\label{qinduction3}
\#\{n\in C : q \not\divides n\} =
\#\{n\in C : \gcd(n,q) < q\} > \sum_{\substack{r < q \\ r \divides q}} \left(\frac qr - 2\right).
\end{equation}
In particular, if \eqref{qinduction3} holds for all $q\geq 2$, then \eqref{casselsassumption3} holds.
\end{remark}

Note that Remark \ref{remarkqinduction3} shows that any set satisfying \eqref{realcassels2} also satisfies \eqref{casselsassumption3}.

\subsection{Proof of Theorem \ref{maintheorem}}
The first main idea of the proof of Theorem \ref{maintheorem} is to combine a lemma of Burr and Erd\doubleacute os with a theorem of Furstenberg, Weiss, and the first-named author. These results are stated as follows:

\begin{lemma}[{\cite[Lemma 3.2]{BurrErdos}}]
\label{lemmasyndetic}
If $B_i \subset \N$ is a set satisfying \eqref{casselsassumption1}, then $\FS(B_i)$ is syndetic.
\end{lemma}

\begin{lemma}[{\cite[Theorem I]{BFW}}]
\label{lemmaBFW}
If $S_1,S_2 \subset \N$ are syndetic sets (or more generally, sets with positive upper Banach density), then $S_1 + S_2$ is a piecewise Bohr set.
\end{lemma}

Note that the converse of Lemma \ref{lemmasyndetic} also holds; see \cite[Theorem 4.1]{AHS}. Since the proof of Lemma \ref{lemmasyndetic} is easy, we include it for completeness:

\begin{proof}[Proof of Lemma \ref{lemmasyndetic}]
Fix $n\in\N$, and define a sequence $(m_j)_1^J$ in $B_i$ recursively using the ``greedy algorithm''
\begin{equation}
\label{greedyalgorithm}
m_j = \max\{m\in B_i\butnot\{m_1,\ldots,m_{j - 1}\} : m_1 + \ldots + m_{j - 1} + m \leq n\},
\end{equation}
where it is understood that the algorithm terminates once the set on the right hand side of \eqref{greedyalgorithm} is empty. Clearly the algorithm always eventually terminates and satisfies $m_1 > m_2 > \ldots > m_J$. Let $D = \{m_1,\ldots,m_J\}$, and let $m = \min(B_i\butnot D)$.
\begin{itemize}
\item[Case 1:] $m = \min(B_i)$. Then since the algorithm terminated at step $J$, we must have $m_1 + \ldots + m_J + m > n$, and thus $n\in \FS(B_i) + \{0,\ldots,\min(B_i)\}$.
\item[Case 2:] $m > \min(B_i)$. Let $j$ be the smallest integer such that $m_j < m$. Since the algorithm selected $m_j$ rather than $m$ at the $j$th step, we must have $m_1 + \ldots + m_{j - 1} + m > n$. On the other hand, letting $\synd = \sup\left\{\big(\ell - \sum\{k\in B_i : k < \ell\}\big) : \ell\in B_i\right\} < \infty$ we have
\[
n \geq m_1 + \ldots + m_J = m_1 + \ldots + m_{j - 1} + \sum\{k\in B_i : k < m\} \geq m_1 + \ldots + m_{n - 1} + m - \synd,
\]
and thus $n + k\in\FS(B_i)$ for some $k = 0,\ldots,\synd$.
\qedhere\end{itemize}
%
\end{proof}

Now let $B_1,B_2,B_3,C$ be as in Theorem \ref{maintheorem}. By Lemma \ref{lemmasyndetic}, the assumption \eqref{casselsassumption1} implies that the sets $\FS(B_1),\FS(B_2),\FS(B_3)$ are syndetic. Let $B_{12} = B_1\cup B_2$. Applying Lemma \ref{lemmaBFW}, we see that $\FS(B_{12}) = \FS(B_1) + \FS(B_2)$ contains a piecewise Bohr set. So there exist $d\in\N$, $\alpha\in \T^d = \R^d/\Z^d$, an open set $U\subset \T^d$, and a thick set $J\subset\N$ such that
\begin{equation}
\label{Udef}
\FS(B_{12})\supset J\cap \{n\in\N : n\alpha\in U\}\neq\emptyset.
\end{equation}
Now let
\begin{equation}
\label{Gdef}
G = \bigcap_{N\in\N} \cl{\FS(\{n\in C, \; n\geq N\})\alpha}.
\end{equation}
\begin{claim}
\label{claim1}
$G$ is a semigroup.
\end{claim}
\begin{subproof}
Fix $n_1,n_2\in G$, $\epsilon > 0$, and $N\in\N$. By definition, there exists $F_1\subset C$ such that $\min(F_1)\geq N$ and $\|\Sigma(F_1)\alpha - n_1\| \leq\epsilon$. By definition, there exists $F_2\subset C$ such that $\min(F_2)\geq \max(F_1) + 1$ and $\|\Sigma(F_2)\alpha - n_2\| \leq\epsilon$. Let $F = F_1\cup F_2$. Then $\min(F)\geq N$ and $\|\Sigma(F)\alpha - (n_1 + n_2)\| \leq 2\epsilon$. Since $\epsilon,N$ were arbitrary, $n_1 + n_2\in G$.
\end{subproof}

Since every compact subsemigroup of a group is itself a group,\Footnote{This fact is proven in \cite[Theorem 1]{Numakura}, but for metric spaces it can be proven more easily as follows: Let $G$ be a compact semigroup of a group, with the group operation written as $+$. Fix $\beta\in G$ and let $(n_k)_1^\infty$ be a sequence such that the sequence $(n_k \beta)_1^\infty$ converges. Without loss of generality suppose that $n_{k + 1} \geq n_k + 2$. Then $-\beta = \lim_{k\to\infty} (n_{k + 1} - n_k - 1)\beta \in G$.} $G$ is a group and thus by the closed subgroup theorem (e.g. \cite[Theorem 20.12]{Lee}), $G$ is an embedded Lie subgroup of $\T^d$, which implies that $G$ takes the form $V/\Z^d + F$, where $V\subset\R^d$ is a rational subspace and $F\subset\T^d$ is a finite subgroup. It follows that $\T^d/G$ is a torus, so there exist continuous homomorphisms $\pi_1,\ldots,\pi_k:\T^d\to\T$ such that $G = \bigcap_1^k \pi_i^{-1}(0)$.

\begin{claim}
\label{claim2}
There exists $q\geq 1$ such that $q\alpha\in G$.
\end{claim}
\begin{subproof}
Suppose not. Then there exists $i = 1,\ldots,k$ such that $\beta = \pi_i(\alpha)$ is irrational. By the assumption \eqref{casselsassumption2}, the series $\sum_{n\in C} \|n\beta\|$ diverges. For each $n\in C$, let $\beta_n\in [-1/2,1/2]$ be a representative of $n\beta\in\T$, so that $\sum_{n\in C} \|n\beta\| = \sum_{n\in C} |\beta_n|$. Let $C_+ = \{n\in C : \beta_n \geq 0\}$, and without loss of generality, suppose that the series $\sum_{n\in C_+} \beta_n$ diverges. Fix $N\in\N$, and let $F_N\subset C_+$ be a finite set which is minimal with respect to the following properties: $\min(F_N)\geq N$ and $\sum_{n\in F_N} \beta_n \geq 1/4$. Then $1/4\leq \sum_{n\in F_N}\beta_n \leq 3/4$, so $\|\Sigma(F_N)\beta\| \geq 1/4$. Since $\T^d$ is compact, we can find a convergent subsequence $\Sigma(F_N)\alpha\to x\in G$; then $\|\pi_i(x)\|\geq 1/4$. But since $x\in G$ and $G = \bigcap_1^k \pi_i^{-1}(0)$, we must have $\pi_i(x) = 0$, a contradiction.
\end{subproof}
Let $q$ be as in Claim \ref{claim2}, and let $H = G + \{0,\ldots,q - 1\}\alpha \supset \N\alpha$. By the assumption \eqref{casselsassumption3}, $\FS(C) + q\Z = \Z$, so $\FS(C)\alpha + G = H$. Then it follows from \eqref{Gdef} that $H$ is contained in the closure of $\FS(C)\alpha$. In particular, $H \subset \FS(C)\alpha + (U\cap H)$, where $U$ is as in \eqref{Udef}. Since $H$ is compact, there exists a finite set $F\subset \FS(C)$ such that $H \subset F\alpha + U$. Now fix $n \geq \max(F)$. Then $n\alpha\in H \subset F\alpha + U$, so there exists $m\in F$ such that $(n - m)\alpha\in U$. If $n - m\in J$, then by \eqref{Udef} we have $n - m\in \FS(B_{12})$ and thus $n\in\FS(B_{12}\cup C)$. So
\[
\FS(B_{12}\cup C) \supset \{n : n - m\in J \all m = 0,\ldots,\max(F)\}.
\]
Since $J$ contains arbitrarily large intervals, so does $\FS(B_{12}\cup C)$. Thus since $\FS(B_3)$ is syndetic, it follows that $\FS(B\cup C) = \FS(B_{12}\cup C) + \FS(B_3)$ is cofinite. This completes the proof of Theorem \ref{maintheorem}.

\subsection{Proof of Theorem \ref{completenesstheorem1fin}}
Let $M = N - 3(a - 1)$ and let
\begin{align*}
B &= \{a^n b_m : n\in\N, \; m = M + 1,\ldots,N\}\\
C &= \{a^n b_m : n\in\N, \; m = 0,\ldots,M\}.
\end{align*}
Then \eqref{casselsassumption1prime} is satisfied with $L = a - 1$, and by Remark \ref{remarkAI} (applied with $C_1 = C$ and $C_2 = a^\N$), \eqref{casselsassumption2} holds for all irrational $\alpha\in\T$. Moreover, by assumption (I) we have $B\cap C = \emptyset$. So to apply Corollary \ref{mainlemmacorollary}, we need to demonstrate \eqref{casselsassumption3}, to which end we will utilize Remarks \ref{remarkqinduction} and \ref{remarkqinduction2}. Thus, we fix $q\geq 2$, aiming to find $r < q$ satisfying \eqref{qinduction}. First, suppose that there is a prime $p$ dividing $q$ which does not divide $a$. By assumption (II), there exists $m = 0,\ldots,M$ such that $p$ does not divide $b_m$. Then for all $n\in\N$, we have $\gcd(a^n b_m,q) < q$, so by Remark \ref{remarkqinduction3} we get \eqref{qinduction}.

On the other hand, suppose that every prime dividing $q$ divides $a$; then $q$ divides $a^n$ for all sufficiently large $n$. Let $n$ be the largest integer such that $q$ does not divide $a^n$. Applying Remark \ref{remarkqinduction2} with $q$ replaced by $a$, by assumption (III) we have $\FS(\{b_m : m = 0,\ldots,M\}) + a\Z = \Z$ and thus
\begin{align*}
\FS(\{a^n b_m : m = 0,\ldots,M\}) + q\Z &= \FS(\{a^n b_m : m = 0,\ldots,M\}) + a^{n + 1}\Z + q\Z\\
&= a^n\Z + q\Z = \gcd(a^n,q)\Z,
\end{align*}
so \eqref{qinduction} holds with $r = \gcd(a^n,q)$.


\subsection{Proof of Theorem \ref{completenesstheorem2}}
We first need to recall a result of Freeman \cite{Freeman2}. Let $\PP_\R$ denote the set of all nonconstant polynomials (with real coefficients), and let $\PP_\R^k$ denote the set of all nonconstant polynomicals of degree $\leq k$. A finite sequence of polynomials $h_1,\ldots,h_s\in\PP_\R$ will be said to satisfy the \emph{irrationality condition} if the set of coefficients of nonconstant terms of the polynomials $h_1,\ldots,h_s$ contains at least two elements which are linearly independent over $\Q$ (cf. \cite[Definition on p.210]{Freeman2}). The sequence will be said to be \emph{positive-definite} if all leading coefficients are positive and all degrees are even.

\begin{theorem}[Corollary of {\cite[Theorem 2]{Freeman2}}]
\label{theoremfreeman}
For all $k\in\N$, there exists $s = s_1(k)\in\N$ such that for every positive-definite sequence $h_1,\ldots,h_s\in\PP_\R^k$ which satisfies the irrationality condition, for all $\epsilon > 0$, there exists $M_0 > 0$ such that for all $\R\ni M\geq M_0$, there exist $z_1,\ldots,z_s\in\Z$ for which
\[
\left|\sum_{i = 1}^s h_i(z_i) - M\right| \leq \epsilon.
\]
Moreover, we may take $s_1(k)$ to satisfy
\[
s_1(k) \sim 4k\log(k).
\]
\end{theorem}

By taking the polynomials $h_1,\ldots,h_s$ to be of the form $h_i(x) = P_i(x^2 + 1)$, we get the following corollary:

\begin{corollary}
\label{corollaryfreeman}
For all $k\in\N$, there exists $s = s_2(k)\in\N$ such that for every sequence $P_1,\ldots,P_s\in\PP_\R^k$ which satisfies the irrationality condition and whose leading coefficients are positive, for all $\epsilon > 0$, there exists $M_0 > 0$ such that for all $\R\ni M\geq M_0$, then there exist $n_1,\ldots,n_s \in\N$ for which
\[
\left|\sum_{i = 1}^s P_i(n_i) - M\right| \leq \epsilon.
\]
Moreover, we may take $s_2(k)$ to satisfy
\[
s_2(k) = 2s_1(k) \sim 8k\log(k).
\]
\end{corollary}

\begin{remark}
\label{remarkgotze}
A result of G\"otze \cite[Corollary 1.4]{Gotze} implies that when $k = 2$, we can get $s_2(2) = 5$ in Corollary \ref{corollaryfreeman}.
\end{remark}

\begin{corollary}
\label{corollaryfreeman2}
Fix $k$ and let $s = s_2(k)$ be as in Corollary \ref{corollaryfreeman}, and fix $a_1,\ldots,a_s\geq 2$, not all powers of the same integer, and $P_1,\ldots,P_s\in\PP_\N^k$ (cf. Theorem \ref{completenesstheorem2}). Then the set
\[
A = \left\{a_1^{P_1(n_1)}\cdots a_s^{P_s(n_s)} : n_1,\ldots,n_s \in \N\right\}.
\]
is sublacunary.
\end{corollary}
\begin{proof}
Apply Corollary \ref{corollaryfreeman} to the sequence of polynomials $\log(a_1)P_1,\ldots\log(a_s)P_s$. Since $a_1,\ldots,a_s$ are not all powers of the same integer and since $P_1,\ldots,P_s$ have integral coefficients, this sequence satisfies the irrationality condition.
\end{proof}

We now begin the proof of Theorem \ref{completenesstheorem2}. Fix $k$, let $s = s_2(k)$, and let
\[
s_0(k) = s_2(k) + 1.
\]
Note that $s_0$ satisfies \eqref{s0bounds}. Fix $a_1,\ldots,a_{s + 1}\geq 2$ and $P_1,\ldots,P_{s + 1}\in\P_\N^k$ such that assumptions (I) and (II) hold. Let
\begin{align*}
B = C_1  &= \left\{\prod_{i = 1}^s a_i^{P_i(n_i)} : n_1,\ldots,n_s \in \N\right\}\\
C_2 &= \{a_{s + 1}^{P_{s + 1}(n)} : n \in \N\}\\
C  &= C_1 C_2 \cup \left\{a_i^{P_i(n)} : n\in\Neur, \;\; i = 1,\ldots,s + 1 \right\}.
\end{align*}
By Corollary \ref{corollaryfreeman2}, $B = C_1$ is sublacunary, so by Remarks \ref{remarksqrt62} and \ref{remarkAI}, \eqref{casselsassumption1prime} and \eqref{casselsassumption2} both hold. Moreover, by assumption (II) we have $B\cap C = \emptyset$.

To demonstrate \eqref{casselsassumption3}, we will use Remark \ref{remarkqinduction3}, so fix $q\geq 2$. Let $p$ be a prime dividing $q$; by assumption (I), we have $p\nmid a_i$ for some $i = 1,\ldots,s + 1$. It follows that $\gcd(a_i^{P_i(n)},q) < q$ for all $n$, demonstrating \eqref{qinduction3}. Thus by Corollary \ref{mainlemmacorollary}, $A = B\cup C$ is complete.

\subsection{Proof of Theorem \ref{theoremzannier}}
For each $n$, find $x_1^{(n)},\ldots,x_k^{(n)} \in A$ satisfying \eqref{zannier} such that $\min(x_1^{(n)},\ldots,x_k^{(n)}) > N_n$, where the sequence $(N_n)_1^\infty$ is chosen recursively so as to satisfy
\begin{equation}
\label{minmaxcond}
\max(x_1^{(n - 1)},\ldots,x_k^{(n - 1)}) < N_n \in A \all n \geq 2.
\end{equation}
Let $F\subset A$ be a finite set such that for all $q = 1,\ldots,b$, $\FS(F)$ intersects every arithmetic progression of the form $q\N + i$ $(0 \leq i < q)$. Let
\begin{align*}
\label{ABdef}
C &= \{x_j^{(n)} : j = 1,\ldots,k , \; n\in\N\}\cup F,&
B &= A\butnot C.
\end{align*}
Since $A$ is sublacunary, the condition \eqref{minmaxcond} implies that $B$ is sublacunary and thus that \eqref{casselsassumption1prime} holds. Fix $\alpha\in\T$ irrational, and let
\[
\epsilon = \min\{\|n\alpha\| : n = 1,\ldots,b\} > 0.
\]
Then for all $n\in\N$, by \eqref{zannier} we have
\[
\left\|\sum_{j = 1}^k z_j x_j^{(n)}\alpha \right\| \geq \epsilon
\]
and thus by the triangle inequality, there exists $j_n = 1,\ldots,k$ such that
\[
\|x_{j_n}^{(n)} \alpha\| \geq \epsilon/(|z_1| + \cdots + |z_k|).
\]
Since $x_{j_n}^{(n)}\in C$, it follows that \eqref{casselsassumption2} holds. Finally, to demonstrate \eqref{casselsassumption3}, we will use Remark \ref{remarkqinduction}, so fix $q\geq 2$. Suppose first that $q > b$. Then for all $n\in\N$, by \eqref{zannier} we have
\[
\sum_{j = 1}^k z_j x_j^{(n)} \notin q\Z
\]
and thus there exists $j_n= 0,\ldots,k$ such that $x_{j_n}^{(n)}\notin q\Z$, i.e. $\gcd(x_{j_n}^{(n)},q) < q$. So by Remark \ref{remarkqinduction3}, \eqref{qinduction} holds. On the other hand, if $2\leq q \leq b$, then the definition of $F$ guarantees that \eqref{qinduction} holds with $r = 1$. Thus by Corollary \ref{mainlemmacorollary}, $A = B\cup C$ is complete.

\subsection{Proof of Theorem \ref{theoremnalpha}}
Let $z_i = (-1)^{k - i} \binom ki$ for all $i = 0,\ldots,k$. Then for all $m\in\N$,
\[
\Delta^k f(m) = \sum_{i = 0}^k z_i f(m + i)
\]
and thus
\[
\left|\sum_{i = 0}^k z_i \lfloor f(m + i)\rfloor \right| \leq b := \sum_{i = 0}^k |z_i| + \sup |\Delta^k f|.
\]
So by Theorem \ref{theoremzannier}, we are done unless for all but finitely many $m\in\N$, we have
\begin{equation}
\label{fmi}
\sum_{i = 0}^k z_i \lfloor f(m + i)\rfloor = 0.
\end{equation}
So by contradiction, suppose that there exists $m_0$ such that \eqref{fmi} holds for all $m\geq m_0$. Let $g$ be the unique polynomial of degree $k - 1$ such that $g(m_0 + i) = \lfloor f(m_0 + i)\rfloor$ for all $i = 0,\ldots,k - 1$. Since $g$ is of degree $k - 1$, for all $m$ we have
\[
\sum_{i = 0}^k z_i g(m + i) = 0,
\]
so a strong induction argument shows that $g(m) = \lfloor f(m)\rfloor$ for all $m\geq m_0$. So
\[
A \supset \{g(n) : n\geq m_0\}
\]
which reduces us to the case considered in \eqref{prezannier}.

\subsection{Proof of Theorem \ref{theoremdPD}}

We begin this proof by introducing a new notation. If $x$ and $y$ are expressions denoting numbers, then $x \lesssim_\times y$ means that $x \leq c y$, where $c > 0$ is a constant independent of $x$ and $y$ (the \emph{implied constant}). The constant $c$ may depend on other variables to be determined from context. We can now state a lemma to be used in the proof:

\begin{lemma}
\label{lemmaMbinom}
Fix $d\in\N$ and $P\in \P_\N^d$. Then for all $n_0,\ldots,n_d\in\N$ distinct, there exist $z_0,\ldots,z_d\in\Z$ such that
\begin{equation}
\label{zibounds}
\max_i |z_i| \lesssim_\times M^{\binom{d + 1}2}
\end{equation}
and
\begin{equation}
\label{ziPi}
0 < \left|\sum_{i = 0}^d z_i P(n_i)\right| \lesssim_\times M^{\binom{d + 1}2},
\end{equation}
where $M = \max_{i,j} |n_j - n_i|$.
\end{lemma}
\begin{proof}
For each $i = 0,\ldots,d$ let $m_i = n_i - n_0$, and write
\[
P(x + m_i) = \sum_{j = 0}^d a_{ij} x^j.
\]
Note that $a_{ij}\in\Z$ and
\begin{equation}
\label{aijbounds}
|a_{ij}| \lesssim_\times m_i^{d - j}.
\end{equation}
Let $D$ denote the determinant of the matrix whose $(i,j)$th entry is $a_{ij}$. By the Vandermonde determinant theorem, $D\neq 0$. Also, the bound \eqref{aijbounds} implies that
\[
|D| \lesssim_\times M^{\binom{d + 1}2}.
\]
Let $z_0,\ldots,z_d$ denote the unique solutions to the equations
\[
\sum_{i = 0}^d a_{ij} z_i = \begin{cases}
D & j = 0\\
0 & j > 0
\end{cases}.
\]
By Cramer's rule, we have $z_i\in\Z$, and combining Cramer's rule with \eqref{aijbounds} gives \eqref{zibounds}. To demonstrate \eqref{ziPi}, we observe that
\begin{align*}
\sum_{i = 0}^d z_i P(n_i)
&= \sum_{i = 0}^d z_i P(n_0 + m_i)
= \sum_{i = 0}^d z_i \sum_{j = 0}^d a_{ij} n_0^j
= \sum_{j = 0}^d n_0^j \begin{cases}
D & j = 0\\
0 & j > 0
\end{cases}
= D.
\qedhere\end{align*}
\end{proof}

\begin{lemma}
\label{lemmadPD}
Fix $d\in\N$ and $P\in \P_\N^d$, and let $D$ be a set such that for some $N_0$,
\begin{equation}
\label{casselsassumption2modv2}
c = \inf_{N\geq N_0} \frac1{N^{1 - \delta}} \#(D\cap [1,N]) > 0,
\end{equation}
where $\delta = 1/\left[1 + 2\binom{d + 1}2\right]$. Let $C = P(D) = \{P(n) : n\in D\}$. Then \eqref{casselsassumption2} holds for all irrational $\alpha\in\T$.
\end{lemma}
\begin{proof}
Fix $\alpha\in\T$ irrational, and let $p/q\in\Q$ be a convergent of the continued fraction expansion of $\alpha$. By standard results in Diophantine approximation \cite[Theorems 13 and 16]{Khinchin_book}, for all $n < q$ we have $\|n\alpha\| \geq 1/(2q)$. Now let $N = (\epsilon q)^{2/(1 - \delta)}$, where $\epsilon > 0$ is a small constant to be chosen below. Assume that $q$ is large enough so that $N\geq N_0$. Then by \eqref{casselsassumption2modv2}, we have $\#(D\cap [1,N]) \geq cN^{1 - \delta}$. Let $(n_k)_1^\infty$ be the unique increasing indexing of $D$. Then
\[
\#\{k : n_{k + 1} \leq N, \; n_{k + 1} - n_k > (d + 1)c^{-1} N^\delta\}
\leq \frac{c}{d + 1} N^{1 - \delta}.
\]
Let $S$ be the set of $k\in\N$ such that $n_{k + d} \leq N$ and $n_{k + i + 1} - n_{k + i} \leq (d + 1)c^{-1} N^\delta$ for all $i = 0,\ldots,d - 1$. Then
\[
\#(S) \geq (cN^{1 - \delta} - d) - d\frac{c}{d + 1}N^{1 - \delta} = \frac{c}{d + 1} N^{1 - \delta} - d.
\]
Fix $k\in S$, and note that $n_{k + d} - n_k \leq (d^2 + d) c^{-1} N^\delta$. By Lemma \ref{lemmaMbinom}, there exist $z_0,\ldots,z_d\in\Z$ such that
\begin{align*}
\max_i |z_i| &\leq K\\
0 < \left|\sum_{i = 0}^d z_i P(n_{k + i})\right| &\leq K
\end{align*}
where
\[
K \asymp_\times (n_{k + d} - n_k)^{\binom{d + 1}2}
\lesssim_\times [(d^2 + d) c^{-1} N^\delta]^{\binom{d + 1}2}
\asymp_\times N^{\binom{d + 1}2 \delta} = \epsilon q.
\]
By choosing $\epsilon$ sufficiently small, we get $K < q$. In particular, since
\[
\left|\sum_{i = 0}^d z_i P(n_{k + i})\right| < q
\]
we have
\[
\left\|\sum_{i = 0}^d z_i P(n_{k + i}) \alpha\right\| \geq \frac1{2q}
\]
and thus
\[
\sum_{i = 0}^d \left\|P(n_{k + i}) \alpha\right\| \geq \frac{1}{2q^2}.
\]
So
\begin{align*}
\sum_{\substack{k \\ n_k \leq N}} \min(1/q^2,\|P(n_k) \alpha\|)
&\geq \frac1{d + 1} \sum_{k\in S} \min\left(1/q^2,\sum_{i = 0}^d \|P(n_{k + i}) \alpha\|\right)\\
\geq \frac1{2(d + 1)q^2} \#(S)
&\geq \frac1{2(d + 1)q^2}\left(\frac{c}{d + 1} N^{1 - \delta} - d\right) = \frac{c}{2(d + 1)^2} - \frac{d}{2(d + 1)q^2}\cdot
\end{align*}
As $q\to\infty$, this inequality implies that the tails of the series $\sum_k \|P(n_k) \alpha\|$ do not converge to zero. It follows that the series \eqref{casselsassumption2} diverges.
\end{proof}

We now begin the proof of Theorem \ref{theoremdPD}. Let $I$ be an infinite subset of $A$ such that for all $q\geq 2$,
\begin{equation}
\label{casselsassumption3modI}
\#\{n\in I : q \not\divides n\} = \infty.
\end{equation}
It is possible to choose $I$ sparse enough so that $A\butnot I$ is a sublacunary set. For each $k\in\N$, let $n_k = P(m_k)\in A\butnot I$ be chosen so that $k^3 \leq m_k < (k + 1)^3$ if possible, with $m_k = \min(D)$ otherwise. Then $D_2 = \{m_k : k\in\N\}$ is a sublacunary set, and so is $B = P(D_2) = \{n_k : k\in\N\}$. On the other hand,
\[
\lim_{N\to\infty} \frac1{N^{1 - \delta}} \#(D_2\cap [1,N]) \leq \lim_{N\to\infty} \frac1{N^{1 - \delta}} (1 + \lceil N^{1/3} \rceil) = 0
\]
and therefore \eqref{casselsassumption2modv2} holds for $D_3 = D\butnot D_2$. In particular, \eqref{casselsassumption2} holds for $C = P(D_3) = A\butnot B$. On the other hand, for all $q\geq 2$, by \eqref{casselsassumption3modI} we have \eqref{qinduction2} and thus by Remarks \ref{remarkqinduction} and \ref{remarkqinduction2}, we have \eqref{casselsassumption3}.

So by Theorem \ref{maintheorem}, $A$ is complete. But if $F$ is any finite subset of $A$, then $A\butnot F$ also satisfies the hypotheses of this corollary, and is therefore complete. Thus $A$ is strongly complete.

\subsection{Proof of Theorem \ref{theoremBEGL}}
Let $F$ be a finite subset of $A$. For each $i = 1,2,3$ let $B_i = S_i^\Neur\butnot F$, and let $C = S_4^\Neur\butnot F$. Fix $i = 1,2,3$ and $n\in B_i$. For each $a\in S_i$ let $m_a\in\Neur$ be the largest integer such that $a^{m_a} < n$, and let $k_a$ be the smallest integer such that $a^m\notin F$ for all $m\geq k_a$. Then by \eqref{a1sum},
\[
n \leq \sum_{a\in S_i} \frac{n}{a - 1} \leq \sum_{a\in S_i} \frac{a^{m_a + 1}}{a - 1} \leq \sum_{a\in S_i} \left[\sum_{m = k_a}^{m_a} a^m + a^{k_a}\right] \leq \sum\{b\in B_i : b < n\} + \sum_{a\in S_i} a^{k_a},
\]
i.e. \eqref{casselsassumption1} holds. Since $C \supset a^{k_a + \Neur} a^\Neur$ for every $a\in S_4$, Remark \ref{remarkAI} implies that \eqref{casselsassumption2} holds. Finally, \eqref{casselsassumption3} follows immediately from Remark \ref{remarkqinduction3} and the assumption that $\gcd(S_4) = 1$. Thus by Theorem \ref{maintheorem}, $A\butnot F$ is complete; since $F$ was arbitrary, $A$ is strongly complete.

\subsection{Proof of Proposition \ref{completenesstheorem0fin}}
We verify \eqref{syndetic2} for $A = a^S b^T$. Let $\synd\in\N$ denote the syndeticity constant of $S$, and let $L = a^\synd - 1$. Fix $N\in\N$ and $m\in T$, and let $n\in S$ be the smallest element such that $a^n b^m > N$. Assuming $N\geq a^{\min(S)} b^m$, this implies that $a^n b^m \in \OC N{(L + 1)N}$. So
\[
\#\big(A\cap \OC N{(L + 1)N}\big) \geq \#\{m\in T : N > a^{\min(S)} b^m\} \tendsto N \#(T) \geq L,
\]
demonstrating \eqref{syndetic2}. So by Remark \ref{remarksyndetic2} and Lemma \ref{lemmasyndetic}, $A$ is complete.

\section{Proofs of dispersing results}
\label{sectiondispersing}

We now state the main theorem which we will use to prove some of our dispersing results, namely Theorems \ref{dispersingtheorem2}, \ref{dispersingtheorem1}, and \ref{dispersingtheorem3}. Theorems \ref{dispersingtheorem0} and \ref{dispersingtheorem4} will be proven separately.

\begin{theorem}
\label{theoremproductsublacunary}
Fix $r\in\N$, and let $(B_i)_1^{2r}$ be a sequence of infinite subsets of $r\N + 1$, of which at least $B_1,\ldots,B_r$ are sublacunary. Then the set
\[
A = \prod_{i = 1}^{2r} B_i
\]
is $1/r$-dispersing.
\end{theorem}

The following lemma will be used in the proof of Theorem \ref{theoremproductsublacunary}.

\begin{lemma}
Let $A$ be a sublacunary set. If $0$ is in the closure of a set $S\subset(0,\infty)$, then $AS$ is dense in $\Rplus$.
\end{lemma}

\begin{proof}
Fix $x > 0$ and $\epsilon > 0$. Let $(n_k)_1^\infty$ be the unique increasing indexing of $A$, and let $k_0$ be large enough so that $|n_{k + 1}/n_k - 1| \leq \epsilon$ for all $k\geq k_0$. Since $0\in\cl S$, there exists $y\in S$ with $0 < y\leq x/n_{k_0}$. Let $k$ be maximal subject to $n_k\leq x/y$. Then
\[
1\leq \frac{x}{n_k y} \leq \frac{n_{k + 1}}{n_k} \leq 1 + \epsilon.
\]
Since $\epsilon$ was arbitrary, we are done.
\end{proof}

Since $\pi(\CO 0\infty) = \T$ (where $\pi:\R\to\T$ is the natural projection), we get:

\begin{corollary}
\label{corollaryAS}
Let $A$ be a sublacunary set. If $0$ is a limit point of a set $S\subset\T$, then $AS$ is dense in $\T$.
\end{corollary}

\begin{proof}[Proof of Theorem \ref{theoremproductsublacunary}]
For each $i = 1,\ldots,r$ let $C_i = B_{r + i}$, and let $A' = \prod_{i = 1}^{r - 1} B_i C_i$. Fix $\alpha\in\T$ irrational.

{\bf Case 1: $\cl{A'\alpha}\cap\Q = \emptyset$.} In this case, letting $k = r - 1$ in the following claim shows that $A' \alpha$ is $1/r$-dense:
\begin{claim}
For all $k = 0,\ldots,r - 1$, there exists $\alpha_k\in\T$ such that
\begin{equation}
\label{alphan}
\alpha_k,\alpha_k + r^{-1}, \ldots, \alpha_k + k r^{-1} \in F_k := \cl{B_1 C_1\cdots B_k C_k \alpha}.
\end{equation}
\end{claim}
\begin{subproof}
For $k = 0$, simply let $\alpha_0 = \alpha$. Fix $k$, and suppose that there exists $\alpha_k$ such that \eqref{alphan} holds. Since $\cl{A'\alpha}$ does not contain any rational, $\alpha_k$ is irrational, so $C_{k + 1}\alpha_k$ is infinite. Since $\T$ is compact, it follows that $0$ is a limit point of $(C_{k + 1} - C_{k + 1})\alpha$. So by Corollary \ref{corollaryAS}, $B_{k + 1}(C_{k + 1} - C_{k + 1})\alpha_k$ is dense in $\T$, and in particular $r^{-1}\in \cl{B_{k + 1}(C_{k + 1} - C_{k + 1})\alpha_k}$. It follows that there exists $\alpha_{k + 1}\in \cl{B_{k + 1} C_{k + 1} \alpha_k}$ such that $\alpha_{k + 1} + r^{-1} \in \cl{B_{k + 1} C_{k + 1} \alpha_k}$. Since $B_{k + 1},C_{k + 1} \subset r\N + 1$, \eqref{alphan} gives
\[
\alpha_{k + 1} + ir^{-1},\alpha_{k + 1} + (i + 1)r^{-1} \in \cl{B_{k + 1} C_{k + 1} \alpha_k} + i r^{-1} \subset \cl{B_{k + 1} C_{k + 1} F_k} = F_{k + 1} \all i = 0,\ldots,k,
\]
which demonstrates \eqref{alphan} for $k = k + 1$.
\end{subproof}

{\bf Case 2: $\cl{A'\alpha}\cap\Q \neq \emptyset$.}
Fix $p/q\in\cl{A'\alpha}$. Then $0\in \cl{q A'\alpha}$. By Corollary \ref{corollaryAS}, $q A' B_r \alpha$ is dense in $\T$. So by elementary topology, $A' B_r \alpha$ is somewhere dense. Multiplying by the infinite set $C_r$ and using the identity $A = A' B_r C_r$ shows that $A\alpha$ is dense, finishing the proof.
\end{proof}

We now use Theorem \ref{theoremproductsublacunary} to prove Theorems \ref{dispersingtheorem2}, \ref{dispersingtheorem1}, and \ref{dispersingtheorem3}.

\subsection{Proof of Theorem \ref{dispersingtheorem2}}
Write $R$ as a disjoint union $R = R'\cup I$, where $R'$ has the same property as $R$ and $I$ is infinite.

Let $r$ be a large prime, and let $k = r - 1$. Write $R'\cap k\N$ as a disjoint union $R'\cap k\N = \bigcup_1^\infty R_i$, where for each $i$, $(R_i/k)\log_b(a)$ is dense mod 1. Then for each $i$, the set
\[
B_i = a^{k\N} b^{R_i} \subset r\N + 1
\]
is sublacunary. So by Theorem \ref{theoremproductsublacunary}, $a^\N b^{\FS(R')} \supset \prod_{i = 1}^{2r} B_i$ is $1/r$-dispersing. Since $r$ was arbitrary, $a^\N b^{\FS(R')}$ is dispersing. Then by elementary topology, $a^S b^{\FS(R')}$ is weakly dispersing. Multiplying by the infinite set $b^I$ finishes the proof.

\subsection{Proof of Theorem \ref{dispersingtheorem1}}
Without loss of generality we can assume that for all $i$, $p$ does not divide $a_i$. Let $s = s_2(k)$ be as in Corollary \ref{corollaryfreeman2}.

Fix $\epsilon > 0$, and let $\ell\in\N$ be large enough so that $r := p^\ell > 1/\epsilon$. For each $j$ let
\[
C_j = \{a_j^{P_j(p^\ell (p - 1) n)} : n\in\N\}
\]
and then let
\[
B_i = \prod_{j = 1}^s C_{si + j}.
\]
By Corollary \ref{corollaryfreeman2}, the sets $(B_i)_1^\infty$ are sublacunary, and from number-theoretical considerations they satisfy $B_i \subset r\N + 1$. So by Theorem \ref{theoremproductsublacunary}, the product $\prod_1^{2r} B_i$ is $1/r$-dispersing. Since $\prod_1^{2r} B_i \subset A_{2rs}$, this completes the proof.

\subsection{Proof of Theorem \ref{dispersingtheorem3}}
Fix $r$ such that $A\cap (r\N + 1)$ is sublacunary, and let $B_1,\ldots,B_{2r}$ be pairwise disjoint sublacunary subsets of $A\cap (r\N + 1)$. Then by Theorem \ref{theoremproductsublacunary}, $\prod_{i = 1}^{2r} B_i$ is $1/r$-dispersing, and thus so is $\FP(A) \supset \prod_{i = 1}^{2r} B_i$. Since $r$ was arbitrary, $\FP(A)$ is dispersing.

The remaining proofs do not use Theorem \ref{theoremproductsublacunary}.

\subsection{Proof of Theorem \ref{dispersingtheorem0}}
Fix $\alpha\in\T$ irrational. Since $S$ is syndetic and $T$ is piecewise syndetic, there exist constants $s,t\in\N$ such that $S' = S + \{0,\ldots,s\}$ is cofinite and $T' = T + \{0,\ldots,t\}$ contains arbitrarily large intervals, say
\[
T' \supset \bigcup_{k = 1}^\infty \{n_k,\cdots,n_k + k\}
\]
for some sequence $n_k\to\infty$. By passing to a subsequence, we may assume that $b^{n_k}\alpha \to \beta$ for some $\beta\in \T$.

Let $A = a^S b^T$ and $A' = a^{S'} b^{T'}$. If we can show that $A' \alpha$ is somewhere dense, then we can complete the proof using elementary topology. Namely,  there exists a finite set $F$ such that $A' = FA$, and thus $\bigcup_{f\in F} f A \alpha$ is somewhere dense. So for some $f\in F$, $f A \alpha$ is somewhere dense and thus $A\alpha$ is somewhere dense.

{\bf Case 1: $\beta$ irrational.} In this case, by Theorem \ref{theoremfurstenberg}, $a^{S'} b^\N \beta$ is dense in $\T$. Fix $n\in S'$ and $m\in\N$. Then for all $k\geq m$,
\[
A' \alpha \ni a^n b^{n_k + m} \alpha \to a^n b^m \beta.
\]
So $\cl{A' \alpha} \supset \cl{a^{S'} b^\N \beta} = \T$.

{\bf Case 2: $\beta$ rational.} After multiplying by the denominator of $\beta$, we may without loss of generality assume that $\beta = 0$, i.e. $b^{n_k} \alpha \to 0$. Fix $\epsilon > 0$, and let $k$ be large enough so that $F = \{0,\ldots,k\} \log_a(b)$ is $\epsilon$-dense mod 1. Then $F + S'$ is $\epsilon$-dense in $\CO c\infty$  for some $c \geq 0$. Choose $\ell\geq k$ large enough so that
\[
\|b^{n_\ell} \alpha\| \leq 1/a^c.
\]
Then $F + S' + \log_a\|b^{n_\ell}\alpha\|$ is $\epsilon$-dense in $\CO 0\infty$. Since the exponential function $x\mapsto a^x$ is $2\log(a)$-Lipschitz on $\OC{-\infty}{\log_a(2)}$, $a^{S'} b^{\{0,\ldots,k\}} \|b^{n_\ell}\alpha\|$ is $2\log(a)\epsilon$-dense in $[1,2]$. But this implies that $A'\alpha \supset a^{S'} b^{\{n_\ell,\ldots,n_\ell + k\}} \alpha$ is $2\log(b)\epsilon$-dense in $\T$. Since $\epsilon$ was arbitrary, this completes the proof.

\subsection{Proof of Theorem \ref{dispersingtheorem4}}

Let $\alpha\in\T$ be Lebesgue random. Then for all $i$, $a_i^\N \alpha$ is dense in $\T$, and in particular $0$ is a limit point of $a_i^\N \alpha$. This will be the only fact about $\alpha$ we need for this proof.

Let $\pi_1,\pi_2:\N\to\N$ be maps such that $\pi_1\times\pi_2:\N\to\N\times\N$ is a bijection. We will define by recursion a sequence $(N_k)_1^\infty$, and then we will show that if
\[
S_i = \bigcup_{k:\pi_1(k) = i} (N_k + \{0,\ldots,\pi_2(k)\}),
\]
then $(A\alpha)'$ (the derived set of $A\alpha$) is $\{0\}$.

Fix $k\in\N$, and suppose that $N_j$ has been defined for all $j < k$. Let
\[
M_k = a_{\pi_1(k)}^{\pi_2(k)} \prod_{j < k} a_{\pi_1(j)}^{N_j + \pi_2(j)}.
\]
By our assumption on $\alpha$, there exists $N_k$ such that $\|a_{\pi_1(k)}^{N_k}\alpha\| \leq (k M_k)^{-1}$. This completes the recursive step.

Now,
\[
A \subset \prod_{k = 1}^\infty \{1\}\cup a_{\pi_1(k)}^{N_k + \{0,\ldots,\pi_2(k)\}}.
\]
So to show that $(A\alpha)' = 0$, it suffices to show that if
\[
n = \prod_{k\in F} a_{\pi_1(k)}^{N_k + s_k} \;\; (\text{$F$ finite, $0\leq s_k\leq \pi_2(k) \all k\in F$}),
\]
then
\[
\|n\alpha\| \leq 1/\max(F).
\]
Indeed, let $k = \max(F)$, and let
\[
m = a_{\pi_1(k)}^{s_k} \prod_{j\in F\butnot\{k\}} a_{\pi_1(j)}^{N_j + s_j}.
\]
Then $m\leq M_k$ and $n = m a_{\pi_1(k)}^{N_k}$. So
\[
\|n\alpha\| \leq m\|a_{\pi_1(k)}^{N_k} \alpha\| \leq 1/k.
\]




\appendix
\section{Growth rate calculations}

\subsection{}
\label{proofanbm}
Fix $a,b\geq 2$ and let $A = \{a^{n^2} b^{m^2} : n,m\in\Neur\}$. In Example \ref{exampleanbm}, we stated that $\FS(A)$ has density zero. Indeed, if $a = b = 2$ then this follows from the fact that infinitely many integers cannot be written as the sum of two squares, so assume that $\max(a,b) \geq 3$. Then for any $N$,
\begin{align*}
\#(A\cap [1,N]) &\leq \#\{(n,m)\in\Neur^2 : n^2 \leq \log_a(N) , m^2 \leq \log_b(N)\}\\
&\leq \sqrt{\log_a(N)}\sqrt{\log_b(N)}\\
\#(\FS(A)\cap [1,N]) &\leq 2^{\#(A\cap [1,N])} \leq \exp\big(\log(2) \sqrt{\log_a(N)}\sqrt{\log_b(N)}\big)
= N^{\sqrt{\log_a(2)\log_b(2)}}.
\end{align*}
Since $a,b\geq 2$ and $\max(a,b)\geq 3$, the exponent is strictly less than one and thus $\FS(A)$ has density zero. In particular $\FS(A)$ is not cofinite, so $A$ is not complete.

We remark that a similar analysis says nothing about the density of the similar-looking set
\[
\FS\left(\left\{2^{\binom n2} 3^{\binom m2} : n,m\in\Neur \right\}\right),
\]
indicating that the issue is somewhat subtle.

\subsection{}
\label{proofs0bounds}
Fix $a_1,\ldots,a_s\geq 2$ pairwise coprime and $P_1,\ldots,P_s\in \PP_\N$ and let
\begin{equation}
\label{completeness2v2}
A = \left\{a_1^{P_1(n_1)}\cdots a_s^{P_s(n_s)} : n_1,\ldots,n_s\in \Neur\right\}.
\end{equation}
Theorem \ref{completenesstheorem2} stated that for all $k$, there exists $s_0 = s_0(k)$ such that if $s\geq s_0$ and $\deg(P_i)\leq k \all i$, then $A$ is complete. In Remark \ref{remarks0bounds}, we stated that $s_0(k) \geq k$, meaning that if $s = k - 1$ and $\deg(P_i) = k\all i$, then $A$ is not complete. In fact, we will prove the following more general result:

\begin{theorem}
\label{completenesstheorem3}
Fix $s\in\N$, $a_1,\ldots,a_s\geq 2$, and $P_1,\ldots,P_s\in\PP_\N$. If
\begin{equation}
\label{completeness3}
\sum_{i = 1}^s \frac{1}{\deg(P_i)} < 1,
\end{equation}
then the set $A$ defined by \eqref{completeness2v2} is not complete.
\end{theorem}
\begin{proof}
Let $C > 0$ be a constant large enough so that for all $i = 1,\ldots,s$ and $s\geq 0$, $P_i(x) \geq (1/C)x^{\deg(P_i)} - C$. Fix
\[
\sum_{i = 1}^s \frac{1}{\deg(P_i)} < \alpha < 1.
\]
Then for all $N\in\N$,
\begin{align*}
\#&\left(\left\{\prod_{i = 1}^s a_i^{P_i(n_i)} : n_1,\ldots,n_s\in\Neur\right\}\cap [1,N]\right) \noreason\\
&\leq \prod_{i = 1}^s \#\{n\in\Neur : a_i^{P_i(n)} \leq N\} \noreason\\
&\leq \prod_{i = 1}^s \#\{n\in\Neur : n^{\deg(P_i)} \leq C\log_{a_i}(N) + C^2\} \noreason\\
&\leq \prod_{i = 1}^s (C\log_{a_i}(N) + C^2 + 1)^{1/\deg(P_i)} \noreason\\
&\leq \log_2(N)^\alpha. \note{if $N$ is sufficiently large}
\end{align*}
Elementary combinatorics then gives
\[
\frac 1N\#\left(\FS\left(\left\{\prod_{i = 1}^s a_i^{P_i(n_i)} : n_1,\ldots,n_s\in\Neur\right\}\right) \cap [1,N]\right) \leq \frac 1N 2^{\log_2(N)^\alpha} \tendsto N 0,
\]
i.e. $\FS(\{\prod_1^s a_i^{P_i(n_i)} : n_1,\ldots,n_s\in\Neur\})$ has density zero, and in particular is not cofinite.
\end{proof}

\bibliographystyle{amsplain}

\bibliography{bibliography}

\providecommand{\bysame}{\leavevmode\hbox to3em{\hrulefill}\thinspace}
\providecommand{\MR}{\relax\ifhmode\unskip\space\fi MR }
\providecommand{\MRhref}[2]{%
  \href{http://www.ams.org/mathscinet-getitem?mr=#1}{#2}
}
\providecommand{\href}[2]{#2}
\begin{thebibliography}{10}

\bibitem{AHS}
C.~Adams{, }III, N.~B. Hindman, and D.~P. Strauss, \emph{Largeness of the set
  of finite products in a semigroup}, Semigroup Forum \textbf{76} (2008),
  no.~2, 276--296.

\bibitem{BerendBoshernitzan}
D.~Berend and M.~D. Boshernitzan, \emph{Densing sets}, Adv. Math. \textbf{115}
  (1995), no.~2, 286--299.

\bibitem{BerendPeres}
D.~Berend and Y.~Peres, \emph{Asymptotically dense dilations of sets on the
  circle}, J. London Math. Soc. (2) \textbf{47} (1993), no.~1, 1--17.

\bibitem{BFW}
V.~Bergelson, H.~Furstenberg, and B.~Weiss, \emph{Piecewise-{B}ohr sets of
  integers and combinatorial number theory}, Topics in discrete mathematics,
  Algorithms Combin., 26, Springer, Berlin, 2006, pp.~13--37.

\bibitem{Birch2}
B.~J. Birch, \emph{Note on a problem of {E}rd{\doubleacute o}s}, Proc.
  Cambridge Philos. Soc. \textbf{55} (1959), 370--373.

\bibitem{Brown}
J.~L. Brown{, }Jr., \emph{Note on complete sequences of integers}, Amer. Math.
  Monthly \textbf{68} (1961), 557--560.

\bibitem{BurrErdos}
S.~A. Burr and P.~Erd{\doubleacute o}s, \emph{Completeness properties of
  perturbed sequences}, J. Number Theory \textbf{13} (1981), no.~4, 446--455.

\bibitem{BEGL}
S.~A Burr, P.~Erd{\doubleacute o}s, R.~L. Graham, and W.~W. Li, \emph{Complete
  sequences of sets of integer powers}, Acta Arith. \textbf{77} (1996), no.~2,
  133--138.

\bibitem{Cassels2}
J.~W.~S. Cassels, \emph{On the representation of integers as the sums of
  distinct summands taken from a fixed set}, Acta Sci. Math. Szeged \textbf{21}
  (1960), 111--124.

\bibitem{Freeman2}
D.~E. Freeman, \emph{Additive inhomogeneous {D}iophantine inequalities}, Acta
  Arith. \textbf{107} (2003), no. 3, 209--244.

\bibitem{Furstenberg3}
H.~Furstenberg, \emph{Disjointness in ergodic theory, minimal sets, and a
  problem in {D}iophantine approximation}, Math. Systems Theory \textbf{1}
  (1967), 1--49.

\bibitem{Glasner}
S.~Glasner, \emph{Almost periodic sets and measures on the torus}, Israel J.
  Math. \textbf{32} (1979), no.~2-3, 161--172.

\bibitem{Gotze}
F.~G\"otze, \emph{Lattice point problems and values of quadratic forms},
  Invent. Math. \textbf{157} (2004), no. 1, 195--226.

\bibitem{Graham}
R.~L. Graham, \emph{Complete sequences of polynomial values}, Duke Math. J.
  \textbf{31} (1964), 275--285.

\bibitem{Hegyvari}
N.~Hegyv\'ari, \emph{On the completeness of an exponential type sequence}, Acta
  Math. Hungar. \textbf{86} (2000), no. 1--2, 127--135.

\bibitem{HKS}
V.~E. Hoggatt{, }Jr., C.~King, and J.~Silver, \emph{Elementary problems and
  solutions: Solutions: E1424}, Amer. Math. Monthly \textbf{68} (1961), no.~2,
  179--180.

\bibitem{Khinchin_book}
A.~Y. Khinchin, \emph{Continued fractions}, The University of Chicago Press,
  Chicago, Ill.-London, 1964.

\bibitem{Lee}
J.~M. Lee, \emph{Introduction to smooth manifolds}, Graduate Texts in
  Mathematics, 218, Springer-Verlag, New York, 2003.

\bibitem{Numakura}
K.~Numakura, \emph{On bicompact semigroups}, Math. J. Okayama Univ. \textbf{1}
  (1952), 99--108.

\bibitem{PorubskyStrauch}
{\v{S}}tefan Porubsk{\'y} and Oto Strauch, \emph{A common structure of
  {$n_k$}'s for which {$n_k\alpha \bmod 1\to x$}}, Publ. Math. Debrecen
  \textbf{86} (2015), no.~3-4, 493--502. \MR{3346100}

\bibitem{RothSzekeres}
K.~F. Roth and G.~Szekeres, \emph{Some asymptotic formulae in the theory of
  partitions}, Quart. J. Math., Oxford Ser. (2) \textbf{5} (1954), 241--259.

\bibitem{Zannier}
U.~M. Zannier, \emph{An elementary proof of some results concerning sums of
  distinct terms from a given sequence of integers}, Studia Sci. Math. Hungar.
  \textbf{27} (1992), no.~1-2, 173--182.

\end{thebibliography}

\end{document}